\documentclass[draft, reqno]{amsart}
\usepackage[all]{xy}                        %

\CompileMatrices                            % Faster

\UseTips                                    % Use

\input xypic
\usepackage[bookmarks=true]{hyperref}       % Hyperref
\usepackage{amssymb,latexsym,amsmath,amscd}
\usepackage{xspace}
\usepackage{enumerate}
\usepackage{graphicx}
\usepackage{vmargin}
\usepackage{todonotes}

\reversemarginpar

\theoremstyle{plain}
\newtheorem{theorem}{Theorem}[section]
\newtheorem*{theorem*}{Theorem}
\newtheorem{proposition}[theorem]{Proposition}

\newtheorem{question}[theorem]{Question}

\theoremstyle{definition}

\newtheorem{remark}[theorem]{Remark}

\newcommand{\enm}[1]{\ensuremath{#1}}          %
\newcommand{\cal}[1]{\mathcal{#1}}

\newcommand{\PP}{\enm{\mathbb{P}}}

\newcommand{\KK}{\enm{\mathbb{K}}}

\newcommand{\Ii}{\enm{\cal{I}}}

\newcommand{\Oo}{\enm{\cal{O}}}

\newcommand{\Ss}{\enm{\cal{S}}}

\newcommand{\Vv}{\enm{\cal{V}}}

\renewcommand{\phi}{\varphi}
\renewcommand{\theta}{\vartheta}
\renewcommand{\epsilon}{\varepsilon}

\begin{document}

\title[Segre varieties]
{Tensor decompositions in rank $+1$}
\author{E. Ballico}
\address{Dept. of Mathematics\\
 University of Trento\\
38123 Povo (TN), Italy}
\email{ballico@science.unitn.it}
\thanks{The author was partially supported by MIUR and GNSAGA of INdAM (Italy).}

\begin{abstract}
We prove (without exceptions) the existence of irredundant tensor decompositions with the number of addenda equal to rank $+1$.
We also discuss the existence of decompositions with more than the tensor rank terms, which are concise, while the original tensor is not concise.
\end{abstract}

\maketitle

\section{Introduction}

Fix $q\in \PP^N$ and a finite subset $A\subset \PP^N$. As in \cite{bbcg} and most references we say that $A$ \emph{irredundantly spans} $q$
if $q\in \langle A\rangle$, where $\langle \ \ \rangle$ denote the linear span, and $q\notin \langle A'\rangle$ for any $A'\subsetneq A$.

Now take an integral and non-degenerate variety $X\subset \PP^N$. For any $q\in \PP^N$ the $X$-rank $r_X(q)$ of $X$ is the minimal cardinality
of a set $A\subset X$ such that $q\in \langle A\rangle$. The minimality assumption implies that $A$ irredundantly spans $q$. Let $\Ss
(X,q)$ be the set of all $A\subset X$ such that $\# A =r_X(q)$ and $q\in \langle A\rangle$. For any positive integer $t$ let
$\Ss (X,q,t)$ denote the set of all $A\subset X$ such that $\# A =t$ and $A$ irredundantly spans
$q$. Obviously $\Ss (X,q,t) =\emptyset$ for all $t\ge N+2$ and, since $X$ is non-degenerate, $\Ss (X,q,N+1) \ne \emptyset$. Obviously  $\Ss (X,q,t) =\emptyset$ for $t<r_X(q)$ and $\Ss (X,q,r_X(q)) =\Ss (X,q)\ne \emptyset$. For many $X$ and $q$ there are gaps, i.e the are integers $q$ and $t$ such that $r_X(q)<t \le N$ and $\Ss (X,q,t)=\emptyset$.
This is not pathological, it is well-known that it may occur even when $X$ is the Veronese varieties (see Remarks \ref{v1} and \ref{v2} for explicit examples).
One of the main results of \cite{bbcg} was that  ``for not too special $q$''  this is not the the case when $X$ is the Segre variety, i.e. \cite[Theorem 3.8]{bbcg} may be restated in the following equivalent way.

\begin{theorem}\label{1i0}
Let $X\subset \PP^N$ be a Segre variety with $\dim X>0$. Fix a linearly independent set $S\subset X$ such that $\#S < N$. For a
general
$q\in
\langle S\rangle$ we have $\Ss (X,q,t)\ne \emptyset$ for any integer $t$ such that $\#S < t\le N+1$.
\end{theorem}
Note that for any linearly independent finite set $S\subset X$ the set $S$ irredundantly spans a general $q\in \langle S\rangle$ (we may take any $q$ in the complement of $\#S$ hyperplanes of $\langle S\rangle$). In this paper
when $t =r_X(q)+1$ we show that there is no exception. We prove the following result.

\begin{theorem}\label{1i01}
Let $X\subset \PP^N$ bea Segre variety with $\dim X>0$.  For any $q\in \PP^N$ we have $\Ss (X,q,r_X(q)+1)\ne \emptyset$.
\end{theorem}

We also prove that concision fails at the level of the tensor rank $+1$, again with no exceptions on $q$. We prove the following result.

\begin{theorem}\label{1i2}
Fix multiprojective spaces $Y\subsetneq W$ and let $\nu: W\to \PP^N$ be the Segre embedding of $W$. Fix $q\in \langle \nu (Y)\rangle$.
Then there is $B\subset W$ such that $B\nsubseteq Y$ and $\nu (B)\in \Ss (\nu (W),r_{\nu(Y)}(q)+1)$.
\end{theorem}

By concision (\cite[Proposition 3.1.3.1]{l}), Theorem \ref{1i2} is false for all tensors $q\in \langle \nu (Y)\rangle$ if we look at decompositions of the tensor $q$ with at most $r_{\nu(Y)}(q)$
terms. Note that concision holds also for Veronese embeddings (\cite[Ex. 3.2.2.2]{l}), hence the difference between Segre
varieties (i.e. tensors and tensor decompositions) and Veronese varieties (i.e. additive decompositions of homogeneous
polynomials) comes only for decompositions with number of components not minimal, but that the corresponding result for
Theorem \ref{1i2} fails for all $q$ (Remarks \ref{v1} and \ref{v2} and Proposition \ref{v3}).

In  section \ref{St} we look at the following problem. 

Suppose you have a Segre variety $X\subset \PP^N$ and a smaller Segre
variety
$X'\subsetneq X$. Take a a tensor
$q\in \langle X'\rangle$ which is concise for $X'$, i.e. there is no smaller Segre variety $X''\subsetneq X'$ such that $q\in \langle \nu (X'')\rangle$. Take any tensor decomposition $A\in \Ss (X,q)$. By concision we have $r_X(q) =r_X(q')$ and $A\subset X$ (\cite[Proposition 3.1.3.1]{l}). A finite set $S\subset X$ is said to be \emph{concise for $X$} if there is no smaller Segre variety $X''\subsetneq X$ such that $S\subset X''$. Given any finite set $S\subset X$ it is easy to determine the minimal Segre variety $X''\subseteq X$ containing $S$ and this Segre variety $X''\subseteq X$ is the only Segre variety $X'\subseteq X$ such that $S\subset X'$ and $S$ is concise for $X'$ (Remark \ref{t1}).

\begin{question}\label{q1}
Let $X\subset \PP^N$ be a Segre variety and $X'\subsetneq X$ a smaller Segre variety. Fix $q\in \langle X'\rangle$ which is concise for $X'$. Compute the minimal integer $r_{X'}(q)$ such that there is $B\in \Ss (X,q,t)$ which is concise for $X$.
\end{question}

When $X\cong X'\times \PP^m$ this is the content of Proposition \ref{t1}.

In section \ref{Ss} we give a few remarks on the Segre-Veronese varieties obtained gluing together ideas and proofs given
for the Segre varieties and the Veronese varieties.

We give the following motivation for the results and problems considered in this paper. Suppose you have a finite set
$S\subset X\subset \PP^r$ with $S$ linearly independent. Write $\PP^r$ as a proiectivization of a vector space $V$. For each
$p\in S$ chose some $v_p\in V\setminus \{0\}$. Call $\KK$ your algebraically closed base field. The vector space $W\subset V$
corresponding to the projective space $\langle S\rangle$ is the set of all $v_q=\sum _{p\in S} c_pv_p$ with $c_p\in \KK$. The
point $q$ associated to $V$ is irredundantly spanned by $S$ if and only if $c_p\ne 0$ for all $p\in S$. Now assume that $X$ is
a Segre variety, so that
$v_q$ is a tensor and $v_q=\sum _{p\in S} c_pv_p$. Having $S$ it is very easy,  effective and cheap to find the minimal Segre
$X'\subseteq X$ containing $S$. Suppose
$X'=X$.
Is it possible to measure how far is $q$ from being certified to be concise, i.e. to give an upper bound on the integer $\dim X-\dim X'$, for instance as a function of $\#S -r_X(q)$? If $\#S > r_X(q)$ we show this in some cases,
e.g., Proposition \ref{t1}. We also point out that for a Segre with at least $4$ factors it is very time consuming to insert in a
computer all entries in fixed bases. Thus tensor decompositions may be used to define the tensor in a cheap way, expecially
if we need many tensors associated to the same set $S$. It is sufficient to precompute each $v_p$ and then for each tensor
it is sufficient to give $\#S$ elements of the field.

\section{Proofs of theorems and examples on the Veronese variety}

\begin{remark}\label{0x0}
Let $Y =\PP^{n_1}\times \cdots \times \PP^{n_k}$, $k\ge 1$, $n_i>0$ for all $i$, be a multiprojective space and let $\nu : Y\to \PP^N$ its Segre embedding. Set $X:= \nu (Y)$. Fix $i\in \{1,\dots ,k\}$
and $o_j\in \PP^{n_j}$, $j
\ne i$. Set $F\subset Y$ be the multiprojective subspace (isomorphic to $\PP^{n_i}$) with $\PP^{n_i}$ as its $i$-th factor and
$\{o_j\}$ as its $j$-th factor. Fix $q\in \PP^N$ and take $A\subset Y$ such that $\nu (A)\in \Ss (X,q)$.

\quad{\bf Claim 1:} We have $\#(A\cap F)\le 1$.

\quad{\bf Proof of Claim 1:} Assume $\#(A\cap F)\ge 2$ and take $u, v\in F$ such that $u\ne v$, say $u =(u_1,\dots ,u_k)$, $v = (v_1,\dots ,v_k)$ with $u_j=v_j =o_j$ for all $j\ne i$
and $v_i\ne u_i$. Set $A':= A\setminus \{u,v\}$. Let $D\subseteq \PP^{n_i}$ be the line spanned by $\{u_i,v_i\}$. Let $L\subset Y$ be the set of all $(x_1,\dots ,x_k)\in Y$ such that $x_j=o_j$ for all $j\ne i$
and $x_i\in D$. The set $\nu (L)\subset \PP^N$ is a line containing $2$ points of $A$. Thus $\langle \nu (L\cup A')\rangle =\langle \nu( A)\rangle$.
Since $\nu (L)$ is a line, there is $a\in L$ such that $q\in \langle \nu (A'\cup \{a\})\rangle$. Thus $r_X(q) < \#A$, a contradiction.
\end{remark}

In the last remark we used in an essential way that $\nu (A)\in \Ss (X,q)$, not just that $\nu (A)$ irredundantly span $q$
(see for instance the proof of Theorem \ref{1i0}).

\begin{proof}[Proof of Theorem \ref{1i01}:]
Set $a:= r_X(q)$. Since $\dim X>0$, we have $a\le N$ (\cite[Proposition 5.1]{lt}). Fix $A\in \Ss (X,q)$ and $a\in A$. Set $A':= A\setminus \{a\}$. Write $X=\nu (Y)$ with $Y$ a multiprojective space,
say $Y =\PP^{n_1}\times \cdots \times \PP^{n_k}$ with $n_i>0$ for all $i$, and $\nu$ the Segre embedding of $Y$. Write $a =(a_1,\dots
,a_k)$.
Set $U:= \langle \nu(A)\rangle$, $E:= \PP^{n_1}\times \{a_2\}\cdots \times \{a_k\}$ and $F:= \nu (E)$. Note that $F$ is a
linear subspace of $\PP^N$.  By construction we have $\nu (a)\in F$. Thus $F\cap U$ is a non-empty linear subspace of $U$. By
Remark \ref{0x0} we have
$A\cap E = \{a\}$, i.e. $\{\nu(a)\} = F\cap \nu (A)$. Set $A':= A\setminus \{a\}$.

\quad (a) Assume $F\nsubseteq U$. Since $A\cap E =\{a\}$, the set $\langle \nu (A')\rangle$ is a linear subspace of $F$ with codimension at least $2$. Take a general line
$L\subseteq
\PP^{n_1}$ containing
$a_1$. Fix a general $(u_1,v_1)\in L\times L$. In particular $u_1\ne v_1$, and neither $u_1$ nor $u_2$ is the first coordinate
$u_1,v_1\in L\setminus
\{a_1\}$ such that
$u_1\ne v_1$. Set $u =(u_1,\dots ,u_k)$ and $v =(v_1,\dots ,v_k)$ with $u_i=v_i=a_i$ for $i=2,\dots ,k$. Set $B:= A'\cup
\{u,v\}$. Since
$a_1\in \langle \{u_1,v_1\}\rangle$ and $\nu (L)$ is a line of $\PP^N$, we have $q\in \langle \nu (B)\rangle$. We may take $u_1$
and $v_1$ with the additional restriction that none of them is the first coordinate of a point of $A'$. With this restriction
we have $\#B=a+1$. Since $\langle \nu (A')\rangle$ is a linear subspace of $F$ with codimension at least $2$, $a\notin A'$ and $L$ is general, we have $\dim
(\langle \nu (B)\rangle \cap F) = \dim  (\langle \nu (A')\rangle \cap F) +2$. Since $\#B =\#A'+2$, $\nu(B)$ is linearly
independent. Since $a_1\in \langle
\{u_1,v_1\}\rangle$ and $\nu(L)$ is a line, we have $U\subseteq \langle \nu (B)\rangle$ and in particular $q\in \langle \nu
(B)\rangle$. To conclude the proof it is sufficient to prove that $\nu(B)$ irredundantly spans $q$. Assume that this is not
true and take a minimal $K\subsetneq B$ with $q\in \langle \nu(K)\rangle$. Since $r_X(q) =a$ and $\#B =a+1$, we have $\#K =a$.
Thus $\nu (K)\in \Ss (X,q)$. Remark \ref{0x0} gives $\{u,v\}\nsubseteq K$. Thus $\#K \cap \{u,v\} =1$, say $K = A'\cup \{u\}$.
Since $q\in U\cap \langle \nu (K)\rangle$, $K\cap A =A'$ and $q\notin \langle \nu (A')\rangle$, we have $\langle K\rangle =U$,
Since $u\in K$, we get $\nu(u)\in U$, contradicting our choice of $u_1$.

\quad (b) By step (a) we may assume $F\subseteq U$ for any choice of the point $a =(a_1,\dots ,a_k)\in A$ and any choice of the index $i\in \{1,\dots ,k\}$ (in step (a) we chose
$i=1$).  Fix a general $o_1\in E$ and write $o:= (o_1,a_2,\dots ,a_k)$ and $A_1:= A'\cup \{o\}$. Since $E\cap A =\{a\}$ by
Remark \ref{0x0}, we get $U =\langle A_1\rangle$. Using $i=2$ and $o$ instead of $a$ we get $U = \langle \nu (A_2)\rangle$,
where $A_2 =A' \cup w$ with $w =(o_1,o_2,a_3,\dots ,k)$ with $o_2$ a general element of $\PP^{n_2}$. In $k-2$ steps using
$i=3,\dots ,k$ we get that $U$ contains a general point of $X$, contradicting the inequality $a\le N$.
\end{proof}

\begin{proof}[Proof of Theorem \ref{1i2}:]  

Write $W =\PP^{n_1}\times \times \cdots \PP^{n_k}$ with $n_i>0$ for all $i$. Up to a permutation of the factors of $W$
we may assume $Y= \PP^{m_1}\times \cdots \times \PP^{m_k}$ with $0\le m_i\le n_i$ for all $i$ and that there is an integer $s\in \{1,\dots ,k\}$ such that
$m_i=0$ if and only if $i>s$. Thus $\dim W -\dim Y = n_1+\cdots +n_k-m_1-\dots -m_s$. Taking a smaller multiprojective space if necessary we may assume $\dim W =\dim Y+1$.
Thus either $k=s+1$, $n_k=1$ and $m_i=n_i$ for all $i$ or $k=s$ and $n_i\ne m_i$ for a unique $i$. In the latter case up to permuting the factors of $W$ we may assume $n_k =m_k-1$. Thus, setting $T:= \prod _{i=1}^{k-1}\PP^{n_i}$, we may unify the cases saying that in both cases there is a hyperplane $H\subset \PP^{n_k}$ such that $Y =T\times H$
and $W =T\times \PP^{n_k}$. Set $a:= r_{\nu (Y)}(q)$ and take $A\subset Y$ such that $\nu(A)\in \Ss (\nu(Y),q)$. By concision we have $a =r_{\nu (W)}(q))$  (\cite[3.2.2.2]{l}).
Set $U:= \langle \nu (A)\rangle$. We modify the proof of Theorem \ref{1i01} in the following way. 
Fix $a\in A$ and write $a=(a_1,\dots ,a_k)$ with $a_k\in H$. Set $A':= A\setminus \{a\}$. Since $A\subset Y$, we have $\langle \nu (A)\rangle \subseteq \langle \nu (Y)\rangle$
Fix $u_k,v_k\in \PP^{n_k}\setminus H$ such that $u_k\ne v_k$ and $a_k\in \langle \{u_k,v_k\}\rangle$. Write $u =(u_1,\dots ,u_k)\in W$ and $v =(v_1,\dots ,v_k)\in W$ with $u_i=v_i=a_i$ for all $i<k$. Set $B:= A'\cup \{u,v\})$. Since $A'\cap \{u,v\} =\emptyset$, we have $\#B =r_{\nu(Y)}(q)+1$. Since $\nu (a)\in \langle \{\nu(u),\nu(v)\}\rangle$.
Write $D:= \{a_1\}\times \cdots \times \{a_{k-1}\}\times \PP^{n_k}$. Since $A\subset Y$, we have $\langle \nu (A)\rangle \subseteq \langle \nu (Y)\rangle$. Thus $\langle \nu (A)\rangle \cap \nu (D)\subseteq \langle \nu (Y)\rangle \cap D = \{a_1\}\times \cdots \times \{a_{k-1}\}\times H$. Part (a) of the proof of Theorem \ref{1i01} shows that $\nu(B)$
irredundantly spans $q$. By construction $B\nsubseteq Y$.
\end{proof}

\begin{remark}\label{v1}
Let $X\subset \PP^N$, $N = -1+\binom{n+d}{n}$, be Veronese variety which is an order $d$ embedding of $\PP^n$, $n\ge 1$, $d\ge 3$. Fix $q\in \PP^N$ with $ r_X(q) \le \lfloor (d+1)/2\rfloor$. Since any $d+1$ points of $X$ are
linearly independent, it is easy the check that $\Ss (X,q,t)=\emptyset$ for all $t$ such that $r_X(q) < t\le d+1-r_X(q)$. 
\end{remark}

\begin{remark}\label{v2}
Let $X\subset \PP^N$, $N = -1+\binom{n+d}{n}$, be Veronese variety which is an order $d$ embedding of $\PP^n$, $n\ge 2$, $d\ge 5$. Fix $q\in \PP^N$ with border
rank $2$ and $r_X(q) >2$. By \cite[Theorem 32]{bgi} we have $r_X(q)=d$. Using \cite[Lemma 3.4]{bgi} it is easy to check that
$\Ss (X,q,t)=\emptyset$ for all $t$ such that $d+1\le t\le 2d-2$.
\end{remark}

\begin{proposition}\label{v3}
Fix integers $d>0$ and $n>m > 0$. Let $\nu : \PP^n\to \PP^N$, $N =\binom{n+d}{n}-1$ be the order $d$ Veronese embedding. Fix an $m$-dimensional linear subspace $M\subset
\PP^n$ and $q\in \langle \nu (M)\rangle$. Take any positive $t$ such that there is $S =\nu (A)\in \Ss (\nu(Y),q,t)$ such that
$M =\langle A\rangle$. Then there is $E=\nu (B) \in \Ss(X,q,t+d(n-m))$ such that $\langle B\rangle =\PP^n$.
\end{proposition}

\begin{proof}
By induction on the integer $n-m$ we reduce to the case $n-m=1$. Fix $a\in A$ and take any line $L\subset \PP^n$ such that
$L\cap M =\{a\}$. Fix a general $G\subset L\setminus \{o\}$ such that $\#G =d+1$ and take $B:= (A\setminus \{a\})\cup G$.
We have $\#B=t+d$. Since any $D\in |\Oo _{\PP^n}(d)|$ contains $L$, we have $\nu (a)\in \langle \nu(B)\rangle$.
Since $A\setminus \{a\}\subset B$, we get $q\in \langle \nu (B)\rangle$. Thus to conclude it is sufficient to prove that
$\nu(B)$ irredundantly spans $q$. Assume that is is not the case, i.e. assume the existence of $B'\subset B$ such that $\#B'
=\#B -1$ and $q\in \langle \nu (B')\rangle$. Set $\{p\}:=  B\setminus B'$.

\quad (a) Assume $p\in G$. Let $\Vv \subseteq |\Oo _L(d)|$ the projectivization of the image of $H^0(\Ii _{A\setminus
\{0\}}(d))$ by the restriction map
$H^0(\Oo _{\PP^n}(d))\to H^0(\Oo _L(d))$. 

\quad {\bf Claim 1:} $\Vv =|\Oo _L(d)|$.

\quad {\bf Proof of Claim 1:} $\Vv$ contains all divisors $a+D$ with $D$ effective of degree $d-1$ (take the image
of all degree $d$ forms on $\PP^n$ with an equation of $M$ as one of their forms. Thus $\Vv$ has at most codimension $1$ in
$|\Oo _L(d)|$ and to prove that $\Vv =|\Oo _L(d)|$ it is sufficient to prove that $a$ is not a base point of $\Vv$. Since $A$
irredundantly spans $q$, there is $T\in |\Oo_M(d)|$ containing $A\setminus \{a\}$, but not containing $a$. For any $o\in
\PP^n\setminus M$ the cone with vertex $o$ and $T$ as its base does not contain $a$, concluding the proof of Claim 1.

By Claim 1 there is $K\in |\Oo _{\PP^n}(d)|$ containing $E$, but not containing $a$. Thus $K\nsupseteq M$.
Thus $K_{|M}$ vanishes on $A\setminus \{a\}$, but not $a$. Since $A$ irredundantly spans $q\in \langle \nu (M)\rangle$, we have
$\langle \nu (A)\rangle =\langle \nu (A')\cup \{q\}\rangle$. Thus $K_{|M}$ shows that $q\notin \langle \nu (E)\rangle$, a
contradiction. 

\quad (b) Assume $p\in A\setminus \{a\}$.  The curve $\nu (L)$ is a degree $d$ rational normal curve in its linear span
and $\langle \nu (L)\rangle \cap \langle \nu(M)\rangle =\{\nu(a)\}$. Since  $\{a\}\cup G\subset
L$ and
$q\langle
\in \nu (M)\rangle$, we get $q\notin \langle \nu (E)\rangle$, a contradiction.
\end{proof}

\section{Concision for tensor decompositions}\label{St}

Let $W:= \PP^{n_1}\times \cdots \times \PP^{n_k}$, $n_i>0$ for all $i$, be a multiprojective space and $\nu : X\to \PP^N$, $N =(n_1+1)\cdots (n_k+1)-1$, its Segre embedding. Set $X:= \nu(W)$.
Let $\pi _i: W\to \PP^{n_i}$ denote the projection of $W$ onto its $i$-th factor.
Take a finite set $S\subset X$ and write $S =\nu (A)$ with $A\subset W$. The minimal Segre subvariety of $X$ containing $S$ is the Segre variety $\nu (\prod _{i=1}^{k} \langle \pi
_i(A)\rangle)$. Thus for any finite $S$ it is easy, quick and very cheap to determines the minimal Segre variety containing
$S$.

\begin{proposition}\label{t1}
Let $Y'$ be a multiprojective space. Fix an integer $m\ge 2$ and fix $o\in \PP^m$. Set $W:= Y'\times \PP^m$ and $Y:= Y'\times \{o\}$. Let $\nu: W\to \PP^N$ be the Segre embedding of $W$. Set $X:= \mu(W)$. Take $q\in \langle \nu (Y)\rangle$.

\quad (a) For any integer $t\le m-1+r_X(q)$ no $B\in \Ss (X,q,t)$ is concise for $X$.

\quad (b) If $q$ is concise for $\nu(Y)$, then there is $B\in \Ss (X,q,r_X(q)+m)$ concise for $X$.
\end{proposition}

\begin{proof}
Let $\pi: W\to \PP^m$ denote the projection onto the last factor of $W$. Assume there is $S\in \Ss (X,q,t)$ which is concise
for
$X$ with $S =\nu (A)$.  Since $S$ is concise for $X$, we have $\langle \pi (A)\rangle =\PP^m$. Thus there is a hyperplane
$H\subset \PP^m$ such that $\#H\cap \pi(A)\ge m$. Set $D:= \pi^{-1}(H)$. $D$ is a hypersurface isomorphic to $Y\times
\PP^{m-1}$ and $\#D\cap A\ge m$. Thus $A':= A\setminus A\cap D$ has (by concision) cardinality
$\le t-m<r_{\nu(Y)}(q)$. Set $U:= \langle \nu (D)\rangle$ and let $\ell : \PP^N\setminus U\to \PP^r$, $r=N-\dim U-1$, denote the linear projection from $U$. 
By the definition of Segre embedding we have $N+1 =(m+1)(\dim \langle \nu(Y)\rangle)$. Since $H\cong Y'\times \PP^{m-1}$, we have $\dim U +1 = m(\dim \langle \nu
(Y)\rangle)$. Hence $r =\dim \langle \nu(Y)\rangle$. By concision for rank $1$ tensors we have $U\cap X =\nu(D)$ and $\langle
\nu (Y)\rangle \cap U=\emptyset$. Thus $\ell _{X\setminus D} :X\setminus D\to \PP^r$ is a morphism. Note that for each
$(u,v)\in Y\times D \setminus D$ we have $\ell (\nu (u,v)) = \nu (u,o)$. Since $q\in \langle \nu (Y)\rangle$ and $\langle \nu
(Y)\rangle \cap U=\emptyset$ we may identify $\PP^r$ with $\langle \nu (Y)\rangle$ and say that, up to this identification, we
have $\ell (q)=q$; alternatively we may say that $\ell (q)$ and $q$ have the same $\nu(Y)$-rank and that $\Ss (\nu (Y),q,t))
=\Ss (\nu(Y),\ell (q),t)$ for any $t$. Since $q\notin U$, $\ell (\langle \nu(B)\rangle \setminus \langle \nu(D)\rangle)$ is a
linear space spanning $\ell (q)$. Since $\langle \nu (B)\rangle$ is spanned by the linearly independent set $\nu(B)$, $\ell
(\langle \nu(B)\rangle \setminus \langle \nu(D)\rangle)$ is spanned by the set $\ell (B\setminus B\cap D)$ with cardinality
$<r_X(q)$, a contradiction. 

\quad (b) Take $E\subset Y$ such that $\nu (E)\in \Ss (\nu(Y),q)$. By concision we have $\#E =r_X(q)$. Write $W = \PP^{n_1}\times \cdots \times \PP^{n_k}$ with $n_k=m$
and $Y =\PP^{n_1}\times \cdots \times \PP^{n_{k-1}}\times \{o\}$. Fix $a =(a_1,\dots ,a_k)\in A$. Obviously $a_k=o$. Fix $m+1$ general points $c_0,\dots ,c_m\in \PP^m$
and set $u_i =(a_1,\dots ,a_{k-1},c_i)$. Set $E:= (A\setminus \{a\})\cup \{u_0,\dots ,u_m)$. We have $\#E = r_X(q)+m$ and $W$ is the minimal multiprojective space containing $E$.
Thus to conclude the proof it is sufficient to prove that $\nu(E)$ minimally spans $q$. Note that $\langle \nu (A)\rangle = \langle \nu (A')\cup \{o\}$. In the set-up of 
For general $c_0,\dots ,c_m$ the point $o$ is not contained in the linear span of any proper subset of $\{u_0,\dots ,u_m\}$. Thus for any proper subset $E'$ of $E$ containing $A\setminus \{o\}$
we have $q\notin \langle \nu (E')\rangle$. Note that $\langle \nu (A)\rangle = \langle \nu (A')\cup \{o\}$. Assume $q\in \langle \nu(J\cup \{u_0,\dots,u_m\})$ with $J\subsetneq A\setminus \{a\}$. Let $H\subset \PP^m$ be the hyperplane spanned by $c_1,\dots ,c_m$. Take $\ell$ as in step (a). Since $\#J \le r_X(q)-2$, we would get that $q$ has rank $<r_X(q)$.
\end{proof}

\section{Segre-Veronese varieties}\label{Ss}
For any multiprojective space $W =\PP^{n_1}\times \cdots \times \PP^{n_k}$ and all positive integers $d_1,\dots ,d_k$
let $\nu_{d_1,\dots ,d_k} W\to \PP^N$, $N= -1 +\prod _{i=1}^{k}\binom{d_i+k_i}{n_i}$, denote the Segre-Veronese embedding of
$Y$ with multidegree $(d_1,\dots ,d_k)$. Since concision holds for Segre-Veronese embeddings, it is natural to consider if
there are irredundantly spanning set with cardinality rank $+1$ or rank $+d_i$.

\begin{remark}\label{sv1}
Fix integers $m>0$, $d>0$ and take $o\in \PP^m$. Set $W = Y\times \PP^m$ . Fix $q\in \langle \nu_{d_1,\dots,d_k,d}(Y\times
\{o\})\rangle$ and call $\rho$ its rank. If $d=1$ there is $S\subset W$ such that $\#S=\rho+1$ and $\nu_{d_1,\dots,d_k,d}(S)$
irredundantly spans
$q$ (just use the proof of Theorem \ref{1i2}). Moreover, if $m\ge 2$ there is no concise $S$ irredundantly spanning $q$ until
$\#S = \rho+m$. Now assume $d>1$. We may repeat the proof of Proposition \ref{v3} with $M =Y\times \{o\}$ and find
that $\Ss (\nu_{d_1,\dots ,d_k,m}(W),q,\rho +dm)\ne \emptyset$.
\end{remark}

\begin{remark}\label{sv2}
Proposition \ref{v3} may be extended with minimal modifications to an inclusion $Y\subset W$ of multiprojective spaces with
the same number of non-trivial factors.
\end{remark}

\end{document}